\documentclass[11pt]{article}
\usepackage{mymacros}

\usepackage[capitalize]{cleveref}
\usepackage{cite}

\usepackage{catchfilebetweentags}
\makeatletter
\def\CatchFBT@Fin@l#1[#2]{%
   \begingroup
      \makeatletter #2%
      \scantokens\expandafter{%
         \expandafter\CatchFBT@tok\expandafter{\the\CatchFBT@tok}}%
      \CatchFBT@IsAToken{#1}
         {\global#1\expandafter{\the\CatchFBT@tok}}
         {\xdef#1{\the\CatchFBT@tok}}%
      \ifx\CatchFBT@tok#1\else\global\CatchFBT@tok{}\fi
   \endgroup
}
\makeatother

\newcommand{\PP}{\bm{P}}

\newcommand{\Pg}{\mathsf{P}}
\newcommand{\Eg}{\mathsf{E}}

\newcommand{\ShG}{{\rm ShG}}
\newcommand{\GFF}{{\rm GFF}}

\newcommand{\phys}{{\rm phys}}

\author{
  Omar Abdelghani\footnote{Courant Institute of Mathematical Sciences, New York University. E-mail: {\tt oa2391@nyu.edu}.}
  \and Roland Bauerschmidt\footnote{Courant Institute of Mathematical Sciences, New York University. E-mail: {\tt bauerschmidt@cims.nyu.edu}.}
  \footnote{Current address: Institut f\"ur Angewandte Mathematik, Universit\"at Bonn. E-mail: {\tt bauerschmidt@uni-bonn.de}.}
  \and Thierry Bodineau\footnote{I.H.E.S., Universit\'e Paris-Saclay, CNRS, Laboratoire Alexandre Grothendieck. 
 E-mail: {\tt bodineau@ihes.fr}.}
  \and Benoit Dagallier\footnote{Department of Mathematics, Imperial College London. E-mail: {\tt b.dagallier@imperial.ac.uk}.}
}
\date{October 2, 2026}

\usepackage[titles]{tocloft}
\title{Log-Sobolev inequality for the sinh-Gordon model}

\begin{document}
\maketitle
\begin{abstract}
    For the sinh-Gordon model (without external mass term),
    we show that the log-Sobolev inequality holds uniformly in the lattice and volume regularisations
    for all parameters in the regime in which the Gaussian multiplicative chaos has second moments.
    This implies that infinite-volume ``massless'' sinh-Gordon measures on $\R^2$ satisfying all Osterwalder--Schrader axioms exist and that they have
    exponential decay of correlations (are ``massive'').

    The proof uses the Polchinski equation method for log-Sobolev inequalities.
    The required estimates on the renormalised potential are established using simple perturbation theoretic bounds,
    valid for absolutely monotone potentials, and an extension of the correlation inequality of Ding--Song--Sun to potentials in the GHS class.
    This inequality is proved also by using a variant of the Polchinski equation through the maximum principle.
\end{abstract}

\begin{center}
  \emph{Dedicated to Chuck Newman on the occasion of his 80th birthday.}
\end{center}

\setcounter{tocdepth}{1}
\tableofcontents

\section{Introduction}

This paper has two parts which can be read independently.
The first part (Section~\ref{sec:LSI} and appendices) concerns 
the log-Sobolev inequality for the sinh-Gordon model---a model in two-dimensional quantum field theory.
Using the log-Sobolev inequality as an a priori bound, we then deduce 
that the sinh-Gordon model can be constructed on the full Euclidean plane and that it has a mass gap.
The second part (Section~\ref{sec:DSS}) extends the Ding--Song--Sun (DSS) inequality to all potentials in the Griffiths--Hurst--Sherman (GHS)
class\footnote{Our proof of the DSS inequality was previously posted as a preprint \cite{2609.08980} which will not be published separately.}
(including in particular the $\cosh$ potential of the sinh-Gordon model).
This inequality is a (black box) ingredient of the proof of the log-Sobolev inequality.

Both parts involve variants of the Polchinski renormalisation group equation
(see \cite{MR4798104} for an introduction), used in different ways.
In the proof of the log-Sobolev inequality, the Polchinski equation is used to provide a multiscale analogue of the Bakry--\'Emery criterion \cite{MR4303014,MR4798104}.
In the proof of the DSS inequality,
the parabolic maximum principle is applied to a variant of the Polchinski equation for the doubled system.

The sinh-Gordon model is a surprisingly interesting example in two-dimensional quantum field theory;
see \cite{MR4258290,MR4430201} for a recent physics reviews, 
and also the introduction of the very recent article by two of us with Hofstetter and Zeitouni \cite{2609.17461} where a quite complete picture for the hierarchical version of the model in the entire
subcritical phase $b^2 \in (0,1)$ is provided.
The construction of the Euclidean model on $\R^2$, implied for $b^2 \in (0,1/2)$ by the results of Section~\ref{sec:LSI} of this paper, had been open.

Rigorous progress was made on the form factor bootstrap program for the sinh-Gordon 
quantum field theory recently \cite{MR4680395,MR4607722},
but the natural construction of the sinh-Gordon model as a probability measure on $\cS'(\R^2)$ is not implied by these results.
Recent probabilistic results  concern the model on a cylinder \cite{MR5055747},
the existence of the free energy of the model \cite{2408.16574},
and tightness of a variant of the hierarchical version of the model \cite{2408.16649}.
In the very recent work of two of us with Hofstetter and Zeitouni \cite{2609.17461} mentioned above,
we obtained a very complete picture for the hierarchical model in the entire subcritical
phase $b^2 \in (0,1)$, including mass gap and log-Sobolev inequality.

Our generalised Ding--Song--Sun inequality of Section~\ref{sec:DSS} implies that,
for all bare potentials in the GHS class \cite{MR395659},
including in particular the $\cosh$ potential of the sinh-Gordon model,
the renormalised potential
defined with Pauli--Villars decomposition is least convex at zero field; for an introduction see again \cite{MR4798104}.
The original DSS inequality is specific to Ising models and
its proof relies on Ising spins taking only two values~\cite{MR4586225}.
By approximation, the DSS inequality extends to models approximable by Ising models such as $\varphi^4$ models.
However, the sinh-Gordon model with arbitrary mass term cannot be approximated by Ising models \cite{MR449366}.

The generalised Ding--Song--Sun inequality extends the conclusion of log-Sobolev inequalities
up to the critical point---defined abstractly by bounded susceptibility---which was established for Ising
and (lattice and continuum) $\varphi^4$ models in \cite{MR4705299,MR4720217},
to all lattice spin models in the GHS class.
Our proof of the log-Sobolev inequality for the sinh-Gordon model in this paper is however \emph{different} from this extension,
in that we prove the boundedness of the log-Sobolev inequality 
without any assumption on the susceptibility. 
We in fact obtain the boundedness of the susceptibility as a consequence of our result and prove the existence of a mass gap.
Moreover, we expect that the log-Sobolev inequality is the main ingredient to show uniqueness of infinite volume measures using the Holley--Stroock uniqueness method~\cite{2504.08606}.

\section{Log-Sobolev inequality for the sinh-Gordon model}
\label{sec:LSI}

\subsection{The Log-Sobolev inequality and its consequences}
\label{sec:LSI-main}

For $b>0$ and $\mu>0$, the lattice regularised sinh-Gordon measure on a discrete torus $\Lambda_\epsilon = \Lambda_{\epsilon,L} \subset \R^2$ of spacing $\epsilon>0$ and arbitrary side length $L>0$
(assumed to be a multiple of $\epsilon$) is defined by
\begin{equation} 
\label{e:shG-measure}
\nu(d\varphi) \propto \exp\qa{\frac{1}{4\pi} (\varphi,\Delta_\epsilon\varphi) - \frac{\mu_\epsilon}{4b^2} \int_{\Lambda_\epsilon} \cosh(2b \varphi(x)) \, dx} \, d\varphi,
\qquad \mu_\epsilon = \mu \epsilon^{2b^2},
\end{equation}
where the discrete Laplacian $\Delta_\epsilon$ is given by 
\begin{equation}
(\Delta_\epsilon\varphi)(x) = \epsilon^{-2}\sum_{y\sim x}[\varphi(y)-\varphi(x)]
,
\end{equation}
where 
$\int_{\Lambda_\epsilon}\cdot \,dx$ is a shorthand for $\epsilon^d\sum_{x\in\Lambda_{\epsilon}}\cdot$ and the inner product $(\cdot,\cdot)$ is the one of $L^2(\Lambda_{\epsilon})$:
\begin{equation}
(\varphi,\psi)
=
\epsilon^d \sum_{x\in\Lambda_\epsilon}\varphi(x)\psi(x)
=
\int_{\Lambda_\epsilon}\varphi(x)\psi(x)\, dx
.
\end{equation}
The Gaussian free field has no mass term and is normalised to have covariance $\sim \log |x|$ (interpreted suitably).
In this convention, the full $L^1$ regime is $b^2<1$ and the $L^2$ regime is $b^2<1/2$.

\begin{theorem} \label{thm:lsi}
For any $b^2\in (0,1/2)$ and $\mu>0$ the sinh-Gordon measure \eqref{e:shG-measure}
has a uniform log-Sobolev constant $\gamma= \gamma(b,\mu)$ satisfying $1/\gamma \lesssim 1+\mu^{-1/(1+b^2)}$:
for any $\epsilon>0$ and $L \geq 1$,
and for any smooth $F: \R^{\Lambda_{\epsilon}} \to \R_+$,
\begin{equation}
\ent_{\nu}(F) \leq \frac{2}{\gamma} \E_\nu\qa{ \|\nabla^\epsilon\sqrt{F}\|_{L^2(\Lambda_\epsilon)}^2},
\end{equation}
where $\nabla^\epsilon=(\epsilon^{-d}\partial_{\varphi(x)})_{x\in\Lambda_{\epsilon}}$ is the $L^2(\Lambda_\epsilon)$ gradient
and the implicit constant in $\lesssim$ depends on $b^2$.
\end{theorem}

We expect with sufficient effort the condition $b^2\in(0,1/2)$ in Theorem~\ref{thm:lsi} could be generalised to the optimal range 
$b^2 \in (0,1)$, using the ideas employed to treat the hierarchical version of the model in \cite{2609.17461},
see Section~\ref{sec:proof}. There is no principal obstruction.

We write $\nu=\nu^{\ShG(b,\mu|L,\epsilon)}$ for the measure \eqref{e:shG-measure} to emphasize the dependence on the parameter $(b,\mu)$
and the regularization parameters $(L,\epsilon)$.

\begin{corollary}\label{cor:tightness}
For any $b^2 \in (0,1/2)$ and $\mu>0$,
the measure $\nu=\nu^{\ShG(b,\mu|L,\epsilon)}$ satisfies
\begin{equation} \label{e:Herbst}
    \E_\nu\qB{e^{(\varphi,f)}} \lesssim e^{\frac{1}{2\gamma}\|f\|_{L^2(\Lambda_\epsilon)}^2}.
\end{equation}
In particular, the family of measures $\nu^{\ShG(b,\mu|L,\epsilon)}$ on $\cS'(\R^2)$
is tight in $\epsilon$ and $L$.
Moreover, 
there is 
exponential decay of correlations: for all $x,y \in \R^2$ with $|x-y| \ll L$,
uniformly in $\epsilon >0$,
\begin{equation} \label{e:massgap}
    \E_{\nu}\qb{(\varphi,f_x)(\varphi,f_y)} \lesssim e^{-c\sqrt{\gamma}|x-y|},
\end{equation}
where $f_x(z) = f(z-x)$ and $f\in C_c^\infty(\R^2)$ is arbitrary.
\end{corollary}

\begin{proof}
The Gaussian integrability follows from the Herbst argument \cite[Theorem~5.4.1]{MR3155209},
and tightness in $\cS'(\R^2)$ then from the Bochner--Minlos theorem \cite[Theorem~I.10]{MR0378670},
using that continuity of the characteristic functional is implied by the integrability estimate.

The exponential correlation decay is a consequence of the spectral gap, implied by the log-Sobolev constant,
along with a propagation speed estimate that follows by convexity. See Appendix~\ref{app:massgap}.
\end{proof}


For finite $L > 0$, the existence of the continuum limit $\epsilon\to 0$ is
a standard consequence from the theory of the Gaussian multiplicative chaos \cite{MR3274356}. 
Corollary~\ref{cor:tightness} implies the existence of subsequential infinite volume limits,
which was open, and provides Gaussian integrability properties and exponential correlation decay
(mass gap for the quantum field theory).

If the infinite-volume limit is defined through domains with zero (rather than periodic) boundary conditions,
then one can use Corollary~\ref{cor:tightness} and the Griffiths inequalities to establish
existence of the infinite-volume limit of the sinh-Gordon model and that the resulting measure is Euclidean invariant. 
By zero boundary conditions we mean that
the finite-volume measure is defined using a Gaussian free field with Dirichlet boundary conditions, while the
$\cosh$ potential keeps the translation-invariant counterterm $\mu_\epsilon$.
We record this as the following corollary, and refer to Appendix~\ref{app:Dirichlet} for more precise definitions
and references.

\begin{corollary} \label{cor:halfdirichlet}
    For any $b^2\in(0,1/2)$,
    the infinite-volume massless sinh-Gordon measure on $\cS'(\R^2)$ constructed using zero boundary conditions exists,
    satisfies all Osterwalder--Schrader axioms, and has a strictly positive mass gap $m_{\rm phys}(b,\mu)>0$.
    Moreover, there is $F(b)>0$ such that
    \begin{equation} \label{e:mass-scaling}
        m_{\rm phys}(b,\mu) = F(b)\mu^{\frac{1}{2b^2+2}}.
    \end{equation}
\end{corollary}

\begin{proof}
The Griffiths inequalities imply that the correlation functions of the zero boundary regularised sinh-Gordon model are dominated by the torus ones.
In particular, the two-point function also decays exponentially for zero boundary conditions, uniformly in the domain, by Corollary~\ref{cor:tightness}.
Similarly, Gaussian integrability holds and the zero boundary finite-volume sinh-Gordon measures are tight.

Using the tightness and exponential decay of the two-point function as input, 
the existence of the continuum and infinite-volume limits, Euclidean invariance, as well as the other properties
follow as for the $\varphi^4_2$ model in~\cite[Chapters VII and VIII]{MR0489552} and \cite{MR0378670}.
Note that our ``zero'' boundary conditions are called ``half-Dirichlet'' boundary conditions in these references, where the ``half'' refers to the
translation-invariant counterterm $\mu_\epsilon$.
The explicit dependence of the mass on $\mu$ in \eqref{e:mass-scaling} is a well-known consequence of scale covariance.
See Appendix~\ref{app:Dirichlet} for further details.
\end{proof}



The dependence of the mass on $\mu$ in \eqref{e:mass-scaling} is simple,
but the function $F(b)$ is complicated. Nonetheless, it is also conjectured to be explicit.

\begin{conjecture}
The mass gap is given by the explicit formula \cite[(2,18)]{MR1625191}:
\begin{equation} {\label{e:FLZZ}}
    F(b) \mu^{\frac{1}{2+2b^2}}
    =
\frac{4\sqrt{\pi}}{\Gamma(\frac{1}{2+2b^2})\Gamma(1+\frac{b^2}{2+2b^2})}
\qa{
    -C_{\rm reg}^{2b^2} b^4\mu \frac{\pi\Gamma(1+b^2)}{\Gamma(-b^2)}}^{\frac{1}{2+2b^2}
}
\end{equation}
where $C_{\rm reg}>0$ is a regularisation-dependent constant.
See also \cite[(2.26)]{MR4258290} and the further discussion in that reference.
\end{conjecture}

Finally, we expect that, by extending the Holley--Stroock uniqueness method developed  for the $\varphi^4_2$ model in \cite{2504.08606},
the log-Sobolev inequality is the main physical ingredient needed to show uniqueness of the invariant measure of the massless sinh-Gordon SPDE on $\R^2$.
Further, we expect that the zero boundary condition infinite-volume limit is this unique invariant measure.



\subsection{Proof of the log-Sobolev inequality}
\label{sec:proof}

Define the regularised massive sinh-Gordon measure with external squared mass $1/t$ for $t >0$ by
the probability density proportional to
\begin{equation} \label{e:shG-measure-t}
\exp\qa{\frac{1}{4\pi} (\varphi,\Delta_\epsilon\varphi) -\frac{1}{2t} (\varphi,\varphi) - \frac{\mu_\epsilon}{4b^2} \int_{\Lambda_{\epsilon}} \cosh(2b \varphi(x)) \, dx},
\qquad \mu_\epsilon = \mu \epsilon^{2b^2},
\end{equation}
and denote its expectation by $\avg{\cdot}_t$.
Let $\chi_t$ be the susceptibility with external mass $1/t$:
\begin{equation}
    \chi_t = \int_{\Lambda_\epsilon} \avg{\varphi(0)\varphi(x)}_t \, dx,
\end{equation}
where $\avg{\cdot}_t$ denotes the expectation with respect to the measure \eqref{e:shG-measure-t}.

\begin{proposition} \label{prop:suscept}
    Let $b^2 < 1/2$ and $\mu>0$. Then there are $c=c(b)>0$ and $C=C(b)>0$ such that
    \begin{equation} \label{e:chit-bd}
        \chi_t \leq
        t - c\mu t^{2+b^2} + C^2\mu^2 t^{3+2b^2}
        = t \qa{1 - c\mu t^{1+b^2} + (C\mu t^{1+b^2})^2 }.
    \end{equation}
    In particular, there is $t_0 \approx \mu^{-1/(1+b^2)}>0$  and $m^2>0$ both independent of $\epsilon$ such that
    \begin{equation}
        \label{e:chit-0}
        \frac{1}{t_0} - \frac{\chi_{t_0}}{t_0^2} \geq m^2 > 0.
    \end{equation}
\end{proposition}

\begin{remark}
Note that convexity of the $\cosh$ potential and the Brascamp--Lieb inequality imply that $1/t -\chi_t/t^2 > 0$ for any $t>0$, but do not give a strictly positive lower bound that is uniform in $\epsilon$. 
We expect a bound similar to~\eqref{e:chit-0} can be shown for $b^2 \in [1/2,1)$, but a different proof is required and the last term in ~\eqref{e:chit-bd} will not be quadratic in $\mu$, but instead appear with an exponent $p\in(1,2)$ depending on $b^2 \in [1/2,1)$:
\begin{equation} \label{e:chit-bd-p}
    \chi_t 
    \leq t \qa{1 - c\mu t^{1+b^2} + (C_p\mu t^{1+b^2})^p }.
\end{equation}
For the hierarchical sinh-Gordon model, such a bound for $b^2 \in [1/2,1)$ is obtained in \cite{2609.17461}.
\end{remark}

Using \eqref{e:chit-0} we now prove Theorem~\ref{thm:lsi}.
The proof of Proposition~\ref{prop:suscept} is given in Section~\ref{sec:suscept}.

\begin{proof}[Proof of Theorem~\ref{thm:lsi}]
We use the Polchinski equation method to prove log-Sobolev inequalities \cite{MR4303014,MR4798104},
applied with Pauli--Villars covariance decomposition.
This decomposition preserves correlation inequalities along
the renormalisation group flow \cite{MR4720217}. For $t>0$, define
\begin{equation}
C_t = (A+1/t)^{-1}, \qquad A = -\frac{1}{2\pi}\Delta_\epsilon
,
\end{equation}
denote its $t$-derivative by $\dot C_t$,
and set
\begin{equation}
     \dot\lambda_t = \frac{1}{t}-\frac{\chi_t}{t^2}.
\label{eq: dot lambda definition}
\end{equation}
The renormalised potential is defined by
\begin{equation}
V_t(\varphi) = -\log \Eg_{C_t}\qa{e^{-V(\varphi+\zeta)}},
\end{equation}
where $\Eg_C$ denotes the Gaussian expectation with covariance $C$ and the ``bare'' potential is given by
\begin{equation}
V(\varphi) = \frac{\mu_\epsilon}{4b^2}\int_{\Lambda_{\epsilon}}\cosh(2b\varphi(x))\, dx;
\end{equation}
see \cite{MR4303014,MR4798104} for background. 
The renormalised measure $\nu_t$ is the probability measure defined as:
\begin{equation}
\nu_t(d\varphi) 
\propto
e^{-S_t(\varphi)}\, d\varphi,\qquad 
S_t(\varphi) 
= 
\frac{1}{2}\big(\varphi,(C_\infty-C_t)^{-1}\varphi\big) + V_t(\varphi)
,
\end{equation}
where we use the interpretation $(C_\infty-C_t)^{-1} = A(tA+1)$.
By construction the sinh-Gordon measure \eqref{e:shG-measure} satisfies $\nu = \nu_0$.
Then, uniformly in the field $\varphi \in \R^{\Lambda_{\epsilon}}$, for any $t>0$,
\begin{equation} \label{e:Vt-convex}
    \dot C_t\He V_t(\varphi)\dot C_t -\frac12 \ddot C_t  \geq \dot \lambda_t \dot C_t
    .
\end{equation}
This follows exactly as in \cite[Section~2]{MR4720217} for the $\varphi^4$ potential,
replacing the Ding--Song--Sun inequality by our generalisation for the cosh potential from Section~\ref{sec:DSS}.
For completeness, we sketch the argument.
By direct computation,
\begin{equation}
\He V_t(\varphi) = C_t^{-1} - C_t^{-1} \cov(\mu_t^\varphi) C_t^{-1},
\end{equation}
where the fluctuation measure $\mu_t^\varphi$ is defined by
\begin{equation}
    \E_{\mu_t^\varphi}[F]
    = e^{+V_t(\varphi)} \Eg_{C_t}\qa{e^{-V_0(\varphi+\zeta)} F(\varphi+\zeta)}.
\end{equation}
Since $C_t^{-1}=A+1/t$ is ferromagnetic,
the DSS inequality (or more precisely inequality \eqref{eq: 2 point DSS estimate})
together with the Perron--Frobenius theorem imply $\cov(\mu_t^\varphi) \leq \chi_t \id$
as quadratic forms.
Since $\dot C_t = t^{-2} C_t^2$ and $-\frac12 \ddot C_t = \frac{1}{t} A C_t \dot C_t$ for the Pauli--Villars decomposition thus
\begin{equation}
\dot C_t \He V_t(\varphi) \dot C_t - \frac12 \ddot C_t \geq \qa{ \frac{1}{t} - \frac{\chi_t}{t^2}} \dot C_t.
\label{eq: convexification}
\end{equation}

As a consequence, we now show that \eqref{e:chit-0} implies
that the measure $\nu=\nu_0$ satisfies the following log-Sobolev inequality:
\begin{equation}
    \ent_{\nu_0}(F)
    \leq 2\qa{\frac{1}{m^2} + t} \E_{\nu_0}\qa{(\nabla^\epsilon \sqrt F)^2}.
\end{equation}
Since $t_0 \approx \mu^{-1/(1+b^2)}$ and $m^2 \approx 1$, 
the log-Sobolev constant satisfies  $1/\gamma \lesssim 1+\mu^{-1/(1+b^2)}$.

Indeed, the Hessian of the action $S_t$ of the renormalised measure $\nu_t$ reads:
\begin{equation}
    \He S_t = (C_\infty-C_t)^{-1} + \He V_t = -\frac12 \dot C_t^{-1}\ddot C_t \dot C_t^{-1}+ \He V_t
\end{equation}
where we used that, for the Pauli--Villars decomposition,
\begin{equation}
    (C_\infty-C_t)^{-1}
    = A(tA+1) = - \frac12 \dot C_t^{-1}\ddot C_t \dot C_t^{-1}.
\end{equation}
The above implies that for $t =t_0$, one has by \eqref{eq: convexification} and \eqref{e:chit-0},
\begin{equation}
    \dot C_t\He S_t\dot C_t \geq m^2\dot C_t .
\end{equation}
The Bakry--\'Emery condition gives the log-Sobolev inequality
\begin{equation}
\ent_{\nu_t}(F)
\leq 
\frac{2}{m^2}\E_{\nu_t}\qa{(\nabla^\epsilon \sqrt{F})_{\dot C_t}^2}
.
\label{eq: LSI renormalised}
\end{equation}
We remark that, in the standard presentation, see for example \cite[Theorem~2.8]{MR4798104},
one has the identity matrix instead of $\dot C_t$. The above simply corresponds to a rescaling.
Indeed, define the measure $\tilde\nu_t(d\psi)$ such that $\E_{\tilde\nu_t}[F(\varphi)] = \E_{\nu_t}[F(\dot C^{-1/2}_t\varphi)]$. 
It has action $\tilde S_t(\varphi) =S_t(\dot C^{1/2}_t\varphi)$ so that
$\He\tilde S_t \geq  m^2 \id$,
and the usual formulation of Bakry-\'Emery with $G(\varphi) = F(\dot C^{1/2}_t\varphi)$ gives:
\begin{equation}
\ent_{\nu_t}(F)
=
\ent_{\tilde \nu_t}(G)
\leq 
\frac{2}{m^2} 
\E_{\nu_t}\qa{(\nabla^\epsilon \sqrt{G})^2}
=
\frac{2}{m^2}\E_{\nu_t}\qa{(\nabla^\epsilon \sqrt{F})_{\dot C_t}^2}
.
\end{equation}

On the other hand, using the Polchinski semigroup $\PP_t$ and the fluctuation measure $\mu_t^\varphi$
which we recall are defined by
\begin{equation}
    \PP_{0,t}F(\varphi)
    = \E_{\mu_t^\varphi}[F]
    = e^{+V_t(\varphi)} \Eg_{C_t}\qa{e^{- V(\varphi+\zeta)} F(\varphi+\zeta)},
\label{eq: fluctuation measure}
\end{equation}
one has the following decomposition of the sinh-Gordon measure \eqref{e:shG-measure}
\begin{equation} \label{e:measure-decomp}
    \E_{\nu_0}[F]
    =\E_{\nu_t}\qb{\PP_{0,t} F}
    =\E_{\nu_t}\qb{\E_{\mu_t^\varphi} F};
\end{equation}
again see \cite{MR4303014,MR4798104} for background.
Under the assumption \eqref{e:Vt-convex},
the Polchinski semigroup satisfies the gradient estimate \cite[(2.46)]{MR4303014}:
uniformly in $\varphi$,
\begin{equation}
    (\nabla^\epsilon \sqrt{\PP_{0,t} F})^2_{\dot C_t} \leq e^{-2\lambda_t} \PP_{0,t}\qa{(\nabla^\epsilon \sqrt F)^2}.
\label{eq: contraction exponentielle}
\end{equation}
The Hamiltonian of the fluctuation measure $\mu_t^\varphi$ defined in \eqref{eq: fluctuation measure} is strictly convex and
bounded from below by the quadratic form $C_t^{-1} \geq 1/t$. Thus by the Bakry-\'Emery Theorem, we get
\begin{equation}
\ent_{\mu_t^\varphi}(F) \leq 2 t \, \E_{\mu_t^\varphi}\qa{(\nabla^\epsilon \sqrt{F})^2}.
\label{eq: LSI fluctuation}
\end{equation}

Putting the bounds \eqref{eq: LSI renormalised}, \eqref{eq: LSI fluctuation}
together and using the measure decomposition \eqref{e:measure-decomp} 
with $t = t_0$ and the corresponding decomposition
of the relative entropy yields the claim:
\begin{align}
    \ent_{\nu_0}(F)
    &= \ent_{\nu_t}(\PP_{0,t} F) +\E_{\nu_t}\qa{\ent_{\mu_t}(F)}
    \nnb
    & \leq \frac{2}{m^2}\E_{\nu_t}\qa{(\nabla^\epsilon \sqrt{\PP_{0,t} F})_{\dot C_t}^2}
 + 2 t \;\E_{\nu_t}\qa{ (\nabla^\epsilon \sqrt{F})^2} \nnb
    &\leq 2\qa{\frac{1}{m^2} + t} \E_{\nu_0}\qa{(\nabla^\epsilon \sqrt F)^2},
\end{align}
where we used \eqref{eq: contraction exponentielle} and $\lambda_t \geq 0$ (to simplify the statement) in the last inequality.
\end{proof}

To complete the proof of Proposition~\ref{prop:suscept}, we begin, in Section~\ref{sec:pert} below,
with a general perturbation theoretic bound for absolutely monotone potentials and
then apply it to the sinh-Gordon model.
The proof of the proposition is then completed in Section~\ref{sec:suscept}.

\subsection{Perturbation theory bound}
\label{sec:pert}

Throughout this section $\Lambda$ is a finite set,
$A$ is ferromagnetic (i.e., $A_{xy} \leq 0$ for $x\neq y$),
and we consider the probability measure with expectation
\begin{equation} \label{e:measure-pert}
    \avg{F} \propto \int F(\varphi) e^{-\frac{1}{2}(\varphi,A\varphi)-V(\varphi)} \, d\varphi
    .
\end{equation}
A function $F: \R^\Lambda \to \R$ is absolutely monotone if it is the restriction of an entire
function  and its power series expansion at $0$ has positive coefficients.

\begin{proposition}
    Let $F$ be an absolutely monotone function with subexponential growth at infinity
    and assume $V(\varphi) = U(\frac12 \varphi^2)$ with $U$ absolutely monotone.
    Then
\begin{equation} \label{e:pert-gen}
\avg{F} \geq \Eg_C[F] - \Eg_C[F;V],
\end{equation}
where $\Eg_C[\cdot]$ and $\Eg_C[\cdot;\cdot]$ denote the Gaussian expectation and covariance with covariance $C=A^{-1}$
and we assume that $A$ is invertible.
\end{proposition}

\begin{proof}
Let
\begin{align}
\avg{F}_\lambda \propto \int e^{-\frac12 (\varphi,A\varphi)  - \lambda V(\varphi)}
F(\varphi)
\, d\varphi,
\end{align}
so that
\begin{equation}
\avg{F}_1
= \avg{F}_0 - \int_0^1 \avg{F;V}_\lambda \, d\lambda.
\end{equation}
By the Brydges--Fr\"ohlich--Sokal generalisation \cite[Proposition~5.1]{MR719815} of Newman's Gaussian inequalities,
applied to the measure with expectation $\avg{\cdot}_\lambda$,
\begin{equation}
\avg{F;V}_\lambda \leq \avg{F;V}^G_\lambda,
\label{e:BFS_gaussian_ineq}
\end{equation}
where the right-hand side refers to the Gaussian expectation with covariance $S_\lambda(x,y)= \avg{\varphi(x)\varphi(y)}_\lambda$.
It follows from the second Griffiths inequality and absolute monotonicity of $V$ that $S_\lambda$ is decreasing in $\lambda\geq 0$, pointwise in $x$ and $y$.
By Wick's formula, the right-hand side of~\eqref{e:BFS_gaussian_ineq} is also decreasing in $\lambda\geq 0$. 
The fact that $S_0=C$ concludes the proof.
\end{proof}

\begin{corollary}\label{cor:perturbation_inequality}
The following correlation inequalities holds for the sinh-Gordon model: for $x\in\Lambda_\epsilon$,
\begin{align}
    \label{e:pert-cosh}
    \frac{1}{2b}
    \avg{\varphi(0)\sinh(2b\varphi(x))}
    &\geq
    e^{2b^2C(0)} C(x)
    \nnb
    &- \frac{\mu_\epsilon}{4b^2} (e^{2b^2 C(0)})^2 C(x)\norm{\cosh(4b^2 C)-1}_1
    \nnb
    &-\frac{\mu_\epsilon}{4b^2} (e^{2b^2 C(0)})^2 (C\ast \sinh(4b^2 C))(x),
\end{align}
where $\|f\|_1 = \int_{\Lambda_\epsilon} |f(x)|\, dx$ and the matrix $A$ and the potential $V$ are defined as in \eqref{e:shG-measure-t}.
\end{corollary}

\begin{proof}
To apply the general correlation bound \eqref{e:pert-gen} to the sinh-Gordon model
and derive \eqref{e:pert-cosh}, define
\begin{equation}
F(\varphi) = \varphi(0) \sinh(\varphi(x)),
\qquad
V(\varphi)=\sum_y \cosh(\varphi(y)).
\end{equation}
Then $F$ is absolutely monotone and $V(\varphi) = U(\frac12 \varphi^2)$ with $U$ absolutely monotone, and
we need the following Gaussian computation.
For a centered Gaussian expectation $\Eg_C$ with covariance $C(x,y)$,
\begin{equation}
    \label{e:pert-gaussian1}
    \Eg_C[F]
    = e^{\frac12 C(x,x)}C(0,x)
\end{equation}
and
\begin{equation} \label{e:pert-gaussian2}
\Eg_C[F;V] = \sum_y e^{\frac12 C(x,x)+\frac12 C(y,y)} \qa{C(0,x)(\cosh(C(x,y))-1) + C(0,y)\sinh(C(x,y))}.
\end{equation}
The claimed estimate follows by applying the general perturbation inequality \eqref{e:pert-gen}
with $\varphi$ replaced by $2b\varphi$ (which has covariance $C$ replaced by $4b^2 C$)
and $V$ replaced by $\frac{\mu_\epsilon}{4b^2} V$ (which simply multiplies the last line by the prefactor).

To complete the proof, we verify the Gaussian identities
\eqref{e:pert-gaussian1} and \eqref{e:pert-gaussian2}.
Using that $\varphi$ is centered, by Girsanov's formula (completion of the square),
\begin{equation}
  \Eg_C[F] = \Eg_C[\varphi(0)\sinh(\varphi(x))] = \Eg_C[{\varphi(0)e^{\varphi(x)}}] = e^{\frac12 C(x,x)} C(0,x),
\end{equation}
which is the first formula \eqref{e:pert-gaussian1}. For the second formula, also observe
\begin{equation}
  \Eg_C[{\cosh(\varphi(y))}] = \Eg_C[{e^{\varphi(y)}}] = e^{\frac12 C(y,y)},
\end{equation}
and that, using Girsanov's formula, 
\begin{equation}
\Eg_C[{\varphi(0)\sinh(\varphi(x))\cosh(\varphi(y))}]
=e^{\frac12 C(x,x)+\frac12 C(y,y)} \qa{C(0,x)\cosh(C(x,y)) + C(0,y)\sinh(C(x,y))}.
\end{equation}
Combining these formulas gives \eqref{e:pert-gaussian2}.
For the last formula, we also record the detailed computation:
\begin{align}
&\Eg_C[{\varphi(0)\sinh(\varphi(x))\cosh(\varphi(y))}]
\nnb
&
= \frac14 \qa{\Eg_C\qa{\varphi(0)\pa{
    e^{\varphi(x)+\varphi(y)}
    +e^{\varphi(x)-\varphi(y)}
    -e^{-\varphi(x)+\varphi(y)}
    -e^{-\varphi(x)-\varphi(y)}
}}}
\nnb
&
= \frac12 \Eg_C\qa{{\varphi(0)\pa{
    e^{\varphi(x)+\varphi(y)}
    +e^{\varphi(x)-\varphi(y)}
}}}
.
\end{align}
By Girsanov, it follows that this is equal to
\begin{align}
&
\frac12
   e^{\frac12 C(x,x)+\frac12 C(y,y)+C(x,y)}(C(0,x)+C(0,y))
\nnb
    &\qquad\qquad
    + \frac12 e^{\frac12 C(x,x)+\frac12 C(y,y)-C(x,y)}(C(0,x)-C(0,y))
    \nnb
    &=
   e^{\frac12 C(x,x)+\frac12 C(y,y)} \qa{C(0,x)\cosh(C(x,y))+C(0,y)\sinh(C(x,y))}
\end{align}
as claimed.
\end{proof}

\subsection{Susceptibility bound}
\label{sec:suscept}

Throughout this section, fix $t>0$,
\begin{equation} \label{e:CS-convention}
S_t(x) = \avg{\varphi(0)\varphi(x)}_t, \qquad C_t(x)=( -\frac{1}{2\pi} \Delta_\epsilon +\frac{1}{t})^{-1}(0,x),
\end{equation}
and the measure~\eqref{e:shG-measure-t} is denoted by $\avg{\cdot}_t$.
The estimates from Section~\ref{sec:pert},
specifically Corollary~\ref{cor:perturbation_inequality},
will be applied with $A^{-1}=C_t$ in \eqref{e:measure-pert}.

\begin{proposition}
The Schwinger--Dyson equation for the regularised sinh-Gordon model is
\begin{equation} \label{e:SD}
    S_t(x)-C_t(x)
    =
    -\frac{\mu_\epsilon}{2b} \int_{\Lambda_\epsilon} C_t(x-z)\avg{\varphi(0)\sinh(2b\varphi(z))}_t \, dz
    .
\end{equation}
\end{proposition}

\begin{proof}
The Schwinger--Dyson equation follows by Gaussian integration by parts:
\begin{align}
    \Eg_C [\varphi(0)\varphi(x) e^{-V}]
    &= \sum_y C (0,y) \Eg_C [\ddp{}{\varphi(y)}(\varphi(x) e^{-V})]
    \nnb
    &= C (0,x) \Eg_C[e^{-V}] - \sum_y C (0,y) \Eg_C[\varphi(0) \ddp{V}{\varphi(y-x)} e^{-V}]
    ,
\end{align}
using translation invariance. Passing to continuum notation gives the claim.
\end{proof}

Below we write $\lesssim$ to denote inequalities up to constants, where implicit constants typically depend on $\delta>0$ if $b\in (0,1/2-\delta)$,
but are uniform in $\epsilon,L$ and $\mu, t$.

\begin{lemma} \label{lem:cov}
The following estimates hold uniformly in $\epsilon \in (0,1]$. For any $t\geq \epsilon^2$,
\begin{equation} \label{e:Ct-L1Linfty}
    \|C_t\|_1 = t, \qquad 
    \|C_t*C_t\|_1 = \|C_t\|_1^2 = t^2,
    \qquad
    \|C_t*C_t\|_\infty \lesssim t,
\end{equation}
and there is a constant $c>0$ such that
\begin{equation} \label{e:expCt}
\epsilon^{2b^2} e^{2b^2 C_t(0)}
= (ct)^{b^2} (1+O(\epsilon^\alpha))
.
\end{equation}
Moreover, if in addition $b^2 < 1/2-\delta$ for some $\delta\in(0,1/2)$,
\begin{equation}
    \|\cosh(4b^2 C_t)-1\|_1 \lesssim t b^4,\qquad
    \|\sinh(4b^2 C_t)\|_1 \lesssim t b^2.
\end{equation}
\end{lemma}

\begin{proof}
    The proof follows from 
    standard estimates for the Green's function on the lattice and
    the bounds $\cosh(x)-1 \leq \frac12 x^2 \cosh(x)$ and $\sinh(x) \leq x\cosh(x)$ for $x\in\R$, 
    which imply that
    \begin{align}
        \|\cosh(4b^2 C_t)-1\|_1
        &\leq \frac12 b^4 \int_{\Lambda_\epsilon} dx \; C_t(x)^2 \cosh(4b^2 C_t(x))
        \lesssim b^4 t ,
        \\
        \|\sinh(4b^2 C_t)\|_1
        &\leq b^2 \int_{\Lambda_\epsilon} dx \; C_t(x) \cosh(4b^2 C_t(x))
        \lesssim b^2 t.
    \end{align}
    Details are given in Appendix~\ref{app:cov}.
\end{proof}

\begin{proof}[Proof of Proposition~\ref{prop:suscept}]
    Define the remainder
    \begin{equation}
        E_t(x) = C_t(x)-S_t(x) \geq 0.
    \end{equation}
    This remainder is nonnegative by the second Griffiths inequality. We need to show the lower bound
    \begin{equation}
        \|E_t\|_1 \geq c\mu t^{2+b^2}-C\mu^2 t^{3+2b^2}.
    \end{equation}
    Indeed, then by positivity of $S_t, C_t, E_t$,
    \begin{equation}
        \chi_t = \|S_t\|_1 = \|C_t\|_1 - \|E_t\|_1 \leq t - c\mu t^{2+b^2} + C\mu^2 t^{3+2b^2}.
    \end{equation}
    By the Schwinger--Dyson equation,
    \begin{equation}
        E_t(x) = \frac{\mu_\epsilon}{2b} \int_{\Lambda_\epsilon} C_t(x-z) \avg{\varphi(0)\sinh(2b\varphi(z))}_t \,dz
        .
    \end{equation}
    The perturbation inequality of Corollary~\ref{cor:perturbation_inequality} implies the lower bound
    \begin{align}
        E_t(x) 
        &\geq \mu_{\epsilon,t}
        \int_{\Lambda_\epsilon} C_t(x-z)C_t(z) \,dz
        \nnb
        &\qquad
        - \frac{\mu_{\epsilon,t}^2}{4b^2} \int_{\Lambda_\epsilon} C_t(x-z)C_t(z) \, dz
        \int_{\Lambda_\epsilon} (\cosh(4b^2 (C_t(y)))-1) \, dy
        \nnb
        &\qquad
        - \frac{\mu_{\epsilon,t}^2}{4b^2} \int_{\Lambda_\epsilon} C_t(x-z) \, dz
        \int_{\Lambda_\epsilon} C_t(z-y) \sinh(4b^2(C_t(y))) \, dy,
    \end{align}
    where $\mu_{\epsilon,t} = \mu_\epsilon e^{2b^2 C_t(0)} \approx \mu t^{b^2}$ by \eqref{e:expCt}.
    Using the estimates from the previous lemma, therefore
    \begin{align}
        \|E_t\|_1
        &\geq \mu_{\epsilon,t}  t^2
        - C\mu_{\epsilon,t}^2 t^3
        \geq c\mu t^{2+b^2}- C\mu^2 t^{3+2b^2}
        .
    \end{align}
    This is the required estimate.
\end{proof}

\section{Ding--Song--Sun inequality for potentials in the GHS class}
\label{sec:DSS}

The journal submission of this paper will contain the proof of the generalised Ding--Song--Sun inequality in this section.
Since the proof was already posted as the arXiv preprint  \cite{2609.08980},
direct inclusion in the arXiv version of this paper is not possible as per arXiv rules.
Therefore we refer for the proof to the preprint \cite{2609.08980} until both papers are merged.
The preprint will not be published separately.

The main conclusion of the DSS inequality
needed in this paper is that sinh-Gordon measures with ferromagnetic coupling matrices satisfy
\begin{equation}
\label{eq: 2 point DSS estimate}
  \forall g\in\R^\Lambda,
  \forall x,y\in \Lambda, \qquad
  \langle \sigma_x;\sigma_y\rangle_g
  \leq \langle \sigma_x;\sigma_y\rangle_0,
\end{equation}
where $\langle \cdot;\cdot\rangle_u$ denotes the covariance with respect to the sine-Gordon measure with external field $u$:
\begin{equation}
\mu^{u}(d\sigma)
\propto
\exp\bigg[-\frac{1}{2}(\sigma,A\sigma) + (\sigma,u) - \sum_{x\in\Lambda}V(\sigma_x)\bigg]
\, d\sigma
,\qquad (u,v) = \sum_{x\in\Lambda} u_xv_x
,
\end{equation}
where $A$ is any ferromagnetic matrix, i.e., $A_{xy} \leq 0$ for $x\neq y$.

\appendix
\section{Green's function estimates}
\label{app:cov}

In this appendix, we prove Lemma~\ref{lem:cov}, namely that for any $t\geq \epsilon^2$:
\begin{equation} \label{e:Green1-app}
    \|C_t\|_1 = t, \qquad 
    \|C_t*C_t\|_1 = \|C_t\|_1^2 = t^2,
    \qquad
    \|C_t*C_t\|_\infty \lesssim t;
\end{equation}
that there is a constant $c>0$ such that
\begin{equation} \label{e:Green2-app}
\epsilon^{2b^2} e^{2b^2 C_t(0)} 
= (ct)^{b^2} (1+O(\epsilon^\alpha))
;
\end{equation}
and that, if in addition $b^2 < 1/2-\delta$ for some $\delta>0$,
\begin{equation} \label{e:Green3-app}
    \|\cosh(4b^2 C_t)-1\|_1 \lesssim tb^4,\qquad
    \|\sinh(4b^2 C_t)\|_1 \lesssim tb^2.
\end{equation}

The estimates are direct calculations if we replace the lattice Green's function $C_t$ on the torus
by the explicit continuum Green's function $\bar C_t(x) = (-\frac{1}{2\pi}\Delta+1/t)^{-1}(0,x)$ on $\R^2$.
Indeed, then
\begin{equation}
    \bar C_t(x) = \bar C_1(x/ \sqrt{t}), \quad \int \bar C_1(x) \, dx =1,
\end{equation}
and there is a continuous function $g$ such that
\begin{equation}
    \bar C_1(x) = \log \frac{1}{|x|} + g(|x|), \qquad \bar C_1(x) \lesssim e^{-|x|} \quad (|x|>1),
\end{equation}
and all estimates with $C_t$ replaced by $\bar C_t$ are straightforward from this.

\begin{proof}[Proof of~\eqref{e:Green1-app}]
    The exact identity $\|C_t\|_1=t$ follows from
    \begin{equation}
        \|C_t\|_1 = (1_{\epsilon,L},(-\frac{1}{2\pi}\Delta+\frac{1}{t})^{-1}1_{\epsilon,L}) = t (1_{\epsilon,L},1_{\epsilon,L}) = t.
    \end{equation}
    where $1_{\epsilon,L}$ is the $L^2$ normalised constant function on $\Lambda_{\epsilon}=\Lambda_{\epsilon,L}$. 
    For the last estimate, observe
    \begin{equation}
        \|C_t*C_t\|_\infty = \sup_{x} \int C_t(x-y)C_t(y) \, dy
        \leq \int C_t(y)^2 \, dy ,
    \end{equation}
    by Cauchy-Schwarz. 
    Using the Parseval identity the right-hand side equals
    \begin{equation}
       \int_{\Lambda^*} \frac{t^2}{(-t \frac{1}{2\pi}\hat\Delta_\epsilon(p) + 1)^2}  \frac{dp}{(2\pi)^2}
       = t\int_{\Lambda^*} \frac{1}{(-t\hat\Delta_\epsilon(p) + 2\pi)^2}  \, t \, dp
       \lesssim t,
    \end{equation}
    where $\Lambda^* = \{p \in \frac{2\pi}{L} \Z^2: -\pi/\epsilon < p_i \leq \pi/\epsilon\}$ is the
    dual torus and
    \begin{equation}
    \hat \Delta_\epsilon(p) 
    = 
    \epsilon^{-2} \sum_{i=1}^2 (2\cos(\epsilon p_i)-2)
    =
    -4\epsilon^{-2} \sum_{i=1}^2 \sin^2(\epsilon p_i/2)
	.    
    \end{equation}
    The convexity bound $\sin^2(x)\geq 4x^2/\pi^2$ valid for all $x\in[-\pi/2,\pi/2]$ shows that the above integral is bounded uniformly in $t,\epsilon$:
    \begin{equation}
    \int_{\Lambda^*} \frac{t\, dp}{(-t\hat\Delta_\epsilon(p) + 2\pi)^2} 
    \lesssim 
    \int_{\Lambda^*} \frac{t\, dp}{(t|p|^2+1)^2} 
    \lesssim 
    \int_{\R^2}\frac{dq}{(|q|^2+1)^2}
    <
    \infty
    .
    \end{equation}
    This completes the proof of \eqref{e:Green1-app}.
\end{proof}

\begin{proof}[Proof of~\eqref{e:Green2-app}]
    For \eqref{e:Green2-app}, we observe the standard estimate, for $\epsilon^2 \leq t \lesssim 1$ and $L \geq 1$:
    \begin{equation}
        C_t(0)
        = \log \frac{\epsilon^2}{t}  + O(1).
    \end{equation}
    which again follows, for example, from the Fourier representation
    \begin{equation}
        \int_{\Lambda^*} \frac{t}{-t \frac{1}{2\pi}\hat \Delta_\epsilon(p) + 1} \frac{dp}{(2\pi)^2},
    \end{equation}
    or alternatively from heat-kernel estimates as in \cite[Proof of Lemma~3.3]{MR4303014}. 
\end{proof}

\begin{proof}[Proof of~\eqref{e:Green3-app}]
    It suffices to show that
    \begin{equation} \label{e:Ct-upper}
        e^{4b^2 C_t(x)} \lesssim (1+\frac{\sqrt{t}}{|x|})^{4b^2},\qquad
        C_t(x) \lesssim
        \pa{1+ \log_+ \frac{\sqrt{t}}{|x|}}e^{-c|x|/\sqrt{t}}.
    \end{equation}
    Using $\cosh(x)-1 \leq \frac{x^2}{2} \cosh(x)$ and $\sinh(x) \leq x \cosh(x)$,
    it then follows that
    \begin{align}
        \|\cosh(4b^2 C_t)-1\|_1
        \leq \frac12 b^4 \int C_t(x)^2 \cosh(4b^2 C_t(x))
        \lesssim b^4 t,
        \\
        \|\sinh(4b^2 C_t)\|_1
        \leq b^2 \int C_t(x) \cosh(4b^2 C_t(x))
        \lesssim b^2 t.
    \end{align}

    For the first bound in \eqref{e:Ct-upper} we use that, uniformly
     $|x| \geq \epsilon$, $t \geq \epsilon^2$, and $L \geq 1$:
    \begin{equation} \label{e:barCC}
        |\bar C_t(x)-C_t(x)| \lesssim 1.
    \end{equation}
    Together with the explicit expression for $\bar C_t$ this implies the first bound in \eqref{e:Ct-upper}:
     \begin{equation}
        e^{4b^2 C_t(x)} \lesssim e^{4b^2 \bar C_t(x)} \lesssim (1+\frac{\sqrt{t}}{|x|})^{4b^2}
        .
    \end{equation}
    To see \eqref{e:barCC}, one can for example use \cite[Theorem~2.1.3]{MR2677157} which compares the lattice and continuum heat kernels
    and that
    \begin{equation}
        C_t(x) = \int_0^\infty e^{-t/s} p_s^{L,\epsilon}(x)\, ds, \qquad \bar C_t(x) = \int_0^\infty e^{-t/s} p_s(x) \, ds.
    \end{equation}
    The reference shows $|p_s^{\epsilon}(x) - p_s(x)| \lesssim \epsilon^2/s^2$ which handles the contribution $s \geq \epsilon^2$.
    Using $p_s^\epsilon(x) \leq \epsilon^{-2}$ for $s \leq \epsilon^2$ and 
    the exponential decay $p_s^\epsilon(x) \lesssim (1/s)e^{-c|x|/\sqrt{s}}$,
    see for example \cite[Lemma~A.1]{MR4303014},    
    to handle the torus periodisation can be used to complete the estimate.

    Similarly, for the second bound,
    \begin{equation}
        C_t(x) = \int e^{-s/t} p_s^{L,\epsilon}(x) \, ds
        \lesssim  \int_0^\infty e^{-s/t} \pa{
            \frac{e^{-c|x|/\sqrt{s}}}{s} + \frac{e^{-cL/\sqrt{s}}}{L^2}
        } \, ds.
    \end{equation}
    For the $L$-independent part,
    the contribution to the integral from $s \leq t$ is bounded by
    \begin{equation}
     \int_0^{t} \frac{e^{-c|x|/\sqrt{s}}}{s} \, ds
     = \int_0^{t/|x|^2} \frac{e^{-c/\sqrt{s}}}{s} \, ds
     = \int_{|x|^2/t}^\infty \frac{e^{-c\sqrt{s}}}{s} \, ds
     \lesssim (1+\log_+ \frac{t}{|x|^2})e^{-c|x|/\sqrt{t}}.
    \end{equation}
    whereas the contribution from $s \geq t$ is less than
    \begin{equation}
     \int_{t}^{\infty} e^{-c|x|/\sqrt{s}}\frac{e^{-s/t}}{s} \, ds
     = \int_{1}^{\infty} e^{-c|x|/\sqrt{st}} \frac{e^{-s}}{s} \, ds
     \lesssim e^{-c|x|/\sqrt{t}}.
    \end{equation}
    The $L$-dependent part is bounded analogously.
    This completes the proof.
\end{proof}

\section{Exponential correlation decay}
\label{app:massgap}

In this appendix, we deduce the exponential correlation decay stated in Corollary~\ref{cor:tightness}. 
The proof uses the spectral gap implied by the log-Sobolev inequality and a simple propagation speed estimate.
This strategy goes back at least to \cite{MR1152374}, see also \cite[Chapter 8]{MR1971582} for a review.
To establish the propagation speed estimate we use a coupling approach, 
going back to \cite{MR1144469} as we learned from \cite{MR4872106}.

Let $W = (W_t(x))$ denote a Wiener process on $L^2(\Lambda_\epsilon)=L^2(\Lambda_{\epsilon,L})$ and consider the corresponding lattice SPDE
\begin{equation}
d\varphi_t
=
\Delta_\epsilon\varphi_t \,dt - V'(\varphi_t) \,dt + \sqrt{2}dW_t
.
\end{equation}
The estimates below are uniform in $0 < \epsilon \leq 1$ and $L\geq 1$ and independent of $V=V_\epsilon$ as long as $V$ is convex. 
Throughout, for a set $A\subset\R^2$, we write $B_R(A)$ for the ball of radius $R>0$ around $A$, to be understood as a subset of $\epsilon\Z^2$ or $\R^2$ depending on the context.

\begin{proposition}
Let $V$ be convex. Let $R>0$ and $f\in C^\infty_c(\R^2)$. 
Let $\tilde \varphi$ denotes a solution to the lattice SPDE defined with another Wiener process $\tilde W$ on the same probability space with $W(x)-\tilde W(x)=0$ for $x\in B_R(\supp f)$, 
and $W,\tilde W$ independent outside of $B_R(\supp f)$.

Then the dynamics started from the same initial condition $\varphi_0=\tilde\varphi_0$ satisfy, uniformly in $\epsilon>0$ and $0 < R \ll L$
(in particular so that the diameter $B_R(\supp f)$ is less than $L$):
\begin{equation}
  \E\qa{ (\varphi_t-\tilde\varphi_t, f)^2 } \lesssim \|f\|_{L^2}^2 e^{-cR^2/t}.
\end{equation}
\end{proposition}

\begin{proof}
Let $E_t = \varphi_t-\tilde\varphi_t$ with $E_0=0$.
Then:
\begin{equation}
dE_t
=
\Delta_\epsilon E_t \,dt -E_t F_t \, dt+ \sqrt{2}(dW_t-d\tilde W_t)
,
\end{equation}
where $F_t(x)=F(\varphi_t(x),\tilde\varphi_t(x))$ is given for $x\in\Lambda_{\epsilon,L}$ by:
\begin{equation}
F(r,z)
=
\begin{cases}
\displaystyle{\frac{V'(r)-V'(z)}{r-z}}
\qquad 
&r\neq z\in\R,\\
V''(r)\qquad &z=r\in\R
.
\end{cases}
\end{equation}
Note that $F\geq 0$ by convexity of $V$.  
Let $(\Pg_{s,t}f)_{s\leq t}$ be the semigroup associated with the following equation on $[0,t]\times\Lambda_\epsilon$:
\begin{equation}
\partial_s \Pg_{s,t}f(x)
=
-\Delta_\epsilon \Pg_{s,t}f(x)+ F_s(x)\Pg_{s,t}f(x)
,\qquad 
\Pg_{t,t}f
= 
f
.
\end{equation}
This semigroup has the explicit Feynman--Kac representation: 
\begin{equation}
\Pg_{s,t} f(x) 
=
\E\bigg[ f(x+\sqrt{2}B_{t-s}) \exp\bigg(-\int_0^{t-s} F_{u+s}\big(x+\sqrt{2}B_{u}\big)\, du\bigg) \bigg]
,
\end{equation}
where $B$ is a discretised Brownian motion on $\Lambda_\epsilon$, with generator $\Delta_\epsilon$. 
It follows that
\begin{align}
  |\Pg_{s,t} f(x)| \leq e^{(t-s)\Delta_\epsilon} |f|(x)
  .
\end{align}
By the Duhamel formula, the solution of the SDE can be written as the stochastic integral
\begin{equation}
 (E_t,f)
 =
\sqrt{2}\int_0^t (\Pg_{s,t} f, dW_s-d\tilde W_s)
,
\end{equation}
and the Ito isometry therefore implies, 
using $W-\tilde W=0$ on $B_R(\supp f)$
and independence outside of $B_R(\supp f)$ so that $d\avg{W-\tilde W}_s(x) = 2 1_{\Lambda_\epsilon\setminus B_R(\supp f)}(x) \,ds$,
\begin{align}
\E\qa{(E_t,f)^2}
&= 2 \int_0^t \int_{\Lambda_\epsilon} d\avg{W-\tilde W}_s(x)  [\Pg_{s,t}f(x)]^2 \, dx
\nnb
&\leq 4 \int_0^t \int_{\Lambda_\epsilon\setminus B_R(\supp f)}  [e^{(t-s)\Delta_\epsilon}|f|](x)^2\, dx.
\end{align}
By writing $p^\epsilon_t(x)= e^{t\Delta_\epsilon}(0,x)$ for the heat kernel on $\Lambda_\epsilon$, the last right-hand side is bounded by
\begin{align}
&\int_0^t\int_{\Lambda_\epsilon \setminus B_R(\supp f)}[ e^{(t-s)\Delta_\epsilon}|f|](x)^2\, dx\, ds
\nnb
&\lesssim 
\|f\|_{L^2}^2\int_{\Lambda_\epsilon \setminus B_R(\supp f)}\int_{\supp(f)} p_{t-s}^\epsilon(x-y)^2\, dx\,dy \, ds
\nnb
&\lesssim 
\|f\|_{L^2}^2 
|\supp(f)|\, 
e^{-cR^2/t}
,
\end{align}
for some $c>0$ independent of $\epsilon$, and using that $L\gg R$ and standard heat kernel estimates.
\end{proof}

\begin{corollary}
For any $x,y\in\Lambda_\epsilon$ with $1 \leq |x-y| \ll L$ and $f\in C^\infty_c(\R^2)$ supported in $B_1(0)$,
\begin{equation}
  \E_\nu\qa{(\varphi,f_x)(\varphi,f_y)}  \lesssim e^{-c\sqrt{\gamma}|x-y|},
\end{equation}
where $f_x= f(x+\cdot)$, $\gamma>0$ is the log-Sobolev constant of $\nu$, and $c>0$ is a universal constant.
\end{corollary}

This inequality gives the claimed decay in Corollary~\ref{cor:tightness}.

\begin{proof}
Let $F_x=(\varphi,f_x)$ and $F_y = (\varphi,f_y)$. Then by invariance of $\nu$ under the dynamics generated by the SPDE with semigroup $T_t$:
\begin{align}
\E_{\nu}\qa{F_xF_y}
&=
\E_{\nu}\qa{T_t(F_xF_y)}
\nnb
&=
\E_{\nu}\qa{(T_tF_x) (T_tF_y)}
+
\E_{\nu}\qa{T_t(F_xF_y)-(T_tF_x) (T_tF_y)}.
\end{align}
By the spectral gap inequality (implied by the log-Sobolev inequality),
\begin{align}
\Big|\E_{\nu}\qa{(T_tF_x) (T_tF_y)}
- \E_{\nu}\qa{F_x}\E_{\nu}\qa{F_y}\Big|
&=
\Big|\E_{\nu}\qa{(T_tF_x) (T_tF_y)}
- \E_{\nu}\qa{T_t F_x}\E_{\nu}\qa{T_t F_y}\Big|
\nnb
&\leq
e^{-2\gamma t}\E_{\nu}[F_0^2]
.
\end{align}
On the other hand, 
let us prove using the propagation speed estimate that there is $c>0$ such that
\begin{equation}
\Big|\E_\nu\qa{T_t(F_xF_y)-(T_tF_x) (T_tF_y)}\Big|
\lesssim 
e^{-c|x-y|^2/t}
.
\end{equation}
Take a Wiener process $W$ on $L^2(\Lambda_\epsilon)$ as well as two further Wiener processes $W^x,W^y$ on $L^2(\Lambda_\epsilon)$, independent from each other and to be chosen below. 
Then, with $\varphi,\varphi^x,\varphi^y$ the solutions to the corresponding SPDEs with the same invariant initial condition, the left-hand side above is given by:
\begin{align}
  &\Big|\E_\nu\qa{ (\varphi_t,f_x)(\varphi_t,f_y) }
  - \E_\nu\qa{ (\varphi^x_t,f_x)(\varphi^y_t,f_y) }\Big|
\nnb
&=
\Big|\E_\nu\qa{ (\varphi_t-\varphi^x_t,f_x)(\varphi_t,f_y) }
+ \E_\nu\qa{ (\varphi^x_t,f_x)(\varphi_t-\varphi^y_t,f_y) }
\Big|
\nnb
&\leq
\E_\nu\qa{ (\varphi,f_0)^2 }^{1/2} \pa{
\E_\nu\qa{ (\varphi_t-\varphi^x_t,f_x)^2 }^{1/2} 
+ \E_\nu\qa{ (\varphi_t-\varphi^y_t,f_y)^2 }^{1/2}
}
,
\end{align}
where we slightly abuse notations above by writing $\E_\nu$ for both the static expectation with respect to $\nu$ and the expectation with respect to any dynamics started from $\nu$.

Let $R < \frac14 |x-y|$ and assume without loss of generality that $R>1$. Given two Wiener processes $B,\tilde B$ such that $W,B,\tilde B$ are all independent,
define $W^x,W^y$ as follows:
\begin{equation}
  W^x
  = 
  1_{B_R(x)} W + 1_{B_R(x)^c} B,
  \qquad 
   W^y
  = 
  1_{B_R(y)}  W + 1_{B_R(y)^c} \tilde B
  .
\end{equation}
Then $W^x,W^y$ are independent since $B_R(\supp(f_x)) \cap B_R(\supp(f_y)) = \varnothing$ by assumption. Moreover,
\begin{equation}
  W-W^x = 0  \qquad \text{on $B_R(x)$}, \qquad
  W-W^y = 0  \qquad \text{on $B_R(y)$}.
\end{equation}
Hence the propagation estimate gives
\begin{equation}
  \E_\nu \qa{ (\varphi_t-\varphi^x_t,f_x)^2 }
  +\E_\nu  \qa{ (\varphi_t-\varphi^y_t,f_y)^2 }
  \lesssim \|f\|_{L^2}^2 e^{-cR^2/t}.
\end{equation}
Taking $t = R/\sqrt{\gamma}$ gives the desired $e^{- c \sqrt{\gamma} R}$ and concludes the proof.
\end{proof}

\section{Zero boundary conditions and infinite-volume limit}
\label{app:Dirichlet}

In this appendix, we further discuss the construction and properties of the infinite-volume sinh-Gordon measure stated in Corollary~\ref{cor:halfdirichlet}.

For an open and bounded domain $D \subset \R^2$, we denote by
$\nu^{\GFF_0(D)}$ the continuum Gaussian free field (GFF) on $D$ with zero boundary conditions on $\partial D$,
and let $M^D(dx)$ denote the symmetrised Gaussian multiplicative chaos (GMC) on $D$ with translation invariant counterterm,
i.e., the limit of
\begin{equation}
    M^{D}_{b,\epsilon}(dx) = \epsilon^{2b^2} \cosh(2b\varphi_\epsilon(x))\, dx
\end{equation}
as a measure, where $\varphi_\epsilon$ is a regularisation of the GFF at scale $\epsilon$;
see for example \cite[Chapter~2]{MR5088010} for background.
The finite-volume zero-boundary sinh-Gordon measure is defined by
\begin{equation}
    \nu^{\ShG_0(b,\mu|D)}(d\varphi) \propto \exp\qa{-\frac{\mu}{4b^2} \int_D M^D_b(dx) }\nu^{\GFF_0(D)}(d\varphi).
\end{equation}
For our purposes, it is sufficient to restrict to the $L^2$ region $b^2 \in (0,1/2)$, where the definition 
and analysis of the GMC is particularly simple.

Provided the domain $D \subset \R^2$ is regular in the sense of \cite[Section~VIII.2]{MR0489552},
for example convex, 
the lattice approximation of this zero boundary condition sinh-Gordon model converges to its continuum version as in \cite[Section~VIII]{MR0489552}.
Because of a regularisation-dependent order one term in the asymptotics of the Green's function (through the Wick ordering),
the normalisation of the parameter $\mu$ can change by a multiplicative factor when using the counterterm $\epsilon^{2b^2}$ rather than $(e^{-C_{\rm reg}}\epsilon)^{2b^2}$.
This also results in the constant $C_{\rm reg}$ in \eqref{e:FLZZ} in our normalisation, but is otherwise unimportant.

Griffiths inequalities imply that the correlation functions of the measures $\nu^{\ShG(b,\mu|D)}$ are increasing
in the (regular) domain $D$ ordered by inclusion, see \cite[Theorem~VIII.20]{MR0489552}. In particular, the infinite-volume
limit exists and is Euclidean invariant. The
other Osterwalder--Schrader axioms with the exception for clustering also follow as in the proof of \cite[Theorem~VIII.15]{MR0489552}.

As discussed in Section~\ref{sec:LSI-main}, exponential decay of the two-point function follows from Corollary~\ref{cor:tightness}.
It then extends to clustering of all correlation functions and the mass gap of the Hamilonian by FKG inequalities, as in \cite[Theorem~VIII.36 and VIII.37]{MR0489552}.

Finally, the explicit dependence of the mass gap $m_{\rm phys}(b,\mu)$ on $\mu$ follows from the scale covariance
of the GMC. Indeed, the continuum Gaussian free field has the scale covariance
\begin{equation}
    \E_{\nu^{\GFF_0(sD)}}[F(\varphi)]
    = \E_{\nu^{\GFF_0(D)}}[F(R_s\varphi)]
    \qquad
    R_s\varphi(x) = \varphi(x/s),
\end{equation}
for $s>0$,
and the GMC satisfies the identity, where $M^D$ is constructed from $\varphi$ and $M^{sD}$ from $R_s\varphi$,
\begin{equation}
    \int_{sD} f(sx) M_b^{sD}(dx) = s^{2+2b^2} \int_D f(x) M_b^D(dx),
\end{equation}
see for example \cite[Theorem~2.8]{MR5088010}.
Therefore
\begin{equation}
     \int F(\varphi) \, \nu^{\ShG_0(b,\mu|sD)}(d\varphi)
     =\int F(R_s\varphi) \,  \nu^{\ShG_0(b,\mu s^{2+2b^2}|D)}(d\varphi),
\end{equation}
and taking the limit $D \to \R^2$ gives
\begin{align}
    \int F(\varphi) \, \nu^{\ShG(b,\mu)}(d\varphi)
    = \int F(R_s\varphi)  \, \nu^{\ShG_0(b,\mu s^{2+2b^2})}(d\varphi).
\end{align}
Since the mass gap (rate of exponential decay) changes by a factor $s$ under rescaling, 
\begin{equation}
    m_{\rm \phys}(b,\mu) = s m_{\rm phys}(b,s^{2+2b^2}\mu) = \mu^{-1/(2+2b^2)} m_{\rm phys}(b,1),
\end{equation}
where in the last equality we chose $s=\mu^{-1/(2+2b^2)}$.

\section*{Acknowledgements}

We thank Michael Hofstetter, Chuck Newman, and Ofer Zeitouni for various related discussions.

This work was supported in part by NSF grant DMS-2348045 and the Simons Collaboration
grant on Probabilistic Paths to Quantum Field Theory.

\section*{AI statement}

No AI assistance was used.

\bibliography{all}

@Article{MR719815,
    author = {Brydges, D.C. and Fr{\"o}hlich, J. and Sokal, A.D.},
    title = "The random-walk representation of classical spin systems and correlation inequalities. {II}. {T}he skeleton inequalities",
    journal = "Commun. Math. Phys.",
    year = "1983",
    volume = "91",
    number = "1",
    pages = "117--139",
    url = "https://projecteuclid.org/getRecord?id=euclid.cmp/1103940478"
}

@Article{MR0378670,
    author = "Guerra, F. and Rosen, L. and Simon, B.",
    title = "The {${\bf P}(\phi)_{2}$} {E}uclidean quantum field theory as classical statistical mechanics. {I}, {II}",
    journal = "Ann. of Math. (2)",
    year = "1975",
    volume = "101",
    pages = "111--189; ibid. (2) 101\ (1975), 191--259",
}

@InCollection{MR1971582,
    Author = "Guionnet, A. and Zegarlinski, B.",
    Title = "Lectures on logarithmic {S}obolev inequalities",
    Booktitle = "S\'eminaire de {P}robabilit\'es, {XXXVI}",
    Publisher = "Springer, Berlin",
    Year = "2003",
    Pages = "1--134",
    Series = "Lecture Notes in Math.",
    Volume = "1801",
    Doi = "10.1007/978-3-540-36107-7_1",
    Url = "https://doi.org/10.1007/978-3-540-36107-7_1"
}

@Article{MR1152374,
    author = "Stroock, D.W. and Zegarli\'nski, B.",
    title = "The equivalence of the logarithmic {S}obolev inequality and the {D}obrushin-{S}hlosman mixing condition",
    journal = "Commun. Math. Phys.",
    year = "1992",
    volume = "144",
    number = "2",
    pages = "303--323",
    url = "https://projecteuclid.org/euclid.cmp/1104249318"
}

@Book{MR3155209,
    author = "Bakry, D. and Gentil, I. and Ledoux, M.",
    title = "Analysis and geometry of {M}arkov diffusion operators",
    publisher = "Springer, Cham",
    year = "2014",
    volume = "348",
    series = "Grundlehren der Mathematischen Wissenschaften",
    isbn = "978-3-319-00226-2; 978-3-319-00227-9",
    doi = "10.1007/978-3-319-00227-9",
    pages = "xx+552",
    url = "https://doi.org/10.1007/978-3-319-00227-9"
}

@Book{MR2677157,
    author = "Lawler, G.F. and Limic, V.",
    publisher = "Cambridge University Press",
    title = "Random walk: a modern introduction",
    year = "2010",
    isbn = "978-0-521-51918-2",
    series = "Cambridge Studies in Advanced Mathematics",
    volume = "123",
    pages = "xii+364",
}

@Book{MR0489552,
    author = "Simon, B.",
    publisher = "Princeton University Press",
    title = "The {$P(\phi )_{2}$} {E}uclidean (quantum) field theory",
    year = "1974",
    note = "Princeton Series in Physics",
    pages = "xx+392",
}

@Article{MR3274356,
    author = "Rhodes, R. and Vargas, V.",
    journal = "Probab. Surv.",
    title = "Gaussian multiplicative chaos and applications: a review",
    year = "2014",
    pages = "315--392",
    volume = "11",
    doi = "10.1214/13-PS218",
    url = "https://doi.org/10.1214/13-PS218"
}

@article{MR395659,
    AUTHOR = "Ellis, R.S. and Monroe, J.L. and Newman, C.M.",
    TITLE = "The {GHS} and other correlation inequalities for a class of even ferromagnets",
    JOURNAL = "Commun. Math. Phys.",
    VOLUME = "46",
    YEAR = "1976",
    NUMBER = "2",
    PAGES = "167--182",
    URL = "http://projecteuclid.org/euclid.cmp/1103899588"
}

@article{MR4303014,
    AUTHOR = "Bauerschmidt, R. and Bodineau, T.",
    TITLE = "Log-{S}obolev inequality for the continuum sine-{G}ordon model",
    JOURNAL = "Comm. Pure Appl. Math.",
    VOLUME = "74",
    YEAR = "2021",
    NUMBER = "10",
    PAGES = "2064--2113",
    DOI = "10.1002/cpa.21926",
    URL = "https://doi.org/10.1002/cpa.21926",
    eprint = "1907.12308",
}

@article{MR4586225,
    AUTHOR = "Ding, J. and Song, J. and Sun, R.",
    TITLE = "A new correlation inequality for {I}sing models with external fields",
    JOURNAL = "Probab. Theory Related Fields",
    VOLUME = "186",
    YEAR = "2023",
    NUMBER = "1-2",
    PAGES = "477--492",
    DOI = "10.1007/s00440-022-01132-1",
    URL = "https://doi.org/10.1007/s00440-022-01132-1"
}

@Article{MR4705299,
    author = "Bauerschmidt, R. and Dagallier, B.",
    journal = "Comm. Pure Appl. Math.",
    title = "Log-{S}obolev inequality for near critical {I}sing models",
    year = "2024",
    number = "4",
    pages = "2568--2576",
    volume = "77",
    eprint = "2202.02301",
}

@Article{MR4720217,
    author = "Bauerschmidt, R. and Dagallier, B.",
    journal = "Comm. Pure Appl. Math.",
    title = "Log-{S}obolev inequality for the {$\varphi^4_2$} and {$\varphi^4_3$} measures",
    year = "2024",
    number = "5",
    pages = "2579--2612",
    volume = "77",
    eprint = "2202.02295",
}

@Article{MR4798104,
    author = "Bauerschmidt, R. and Bodineau, T. and Dagallier, B.",
    journal = "Probab. Surv.",
    title = "Stochastic dynamics and the {P}olchinski equation: {A}n introduction",
    year = "2024",
    pages = "200--290",
    volume = "21",
    doi = "10.1214/24-ps27",
    url = "https://doi.org/10.1214/24-ps27",
    eprint = "2307.07619"
}

@article{MR4872106,
    AUTHOR = "Gubinelli, M. and Hofmanov\'a, M. and Rana, N.",
    TITLE = "Decay of correlations in stochastic quantization: the exponential {E}uclidean field in two dimensions",
    JOURNAL = "Stoch. Partial Differ. Equ. Anal. Comput.",
    VOLUME = "13",
    YEAR = "2025",
    NUMBER = "1",
    PAGES = "107--145",
    DOI = "10.1007/s40072-024-00328-x",
    URL = "https://doi.org/10.1007/s40072-024-00328-x"
}

@Article{2504.08606,
    author = "Bauerschmidt, R. and Dagallier, B. and Weber, H.",
    title = "{H}olley--{S}troock uniqueness method for the {$\varphi^4_2$} dynamics",
    archiveprefix = "arXiv",
    eprint = "2504.08606",
    publisher = "arXiv",
    note = "Preprint, arXiv:2504.08606"
}

@article{MR4258290,
    AUTHOR = "Konik, R. and L\'ajer, M. and Mussardo, G.",
    TITLE = "Approaching the self-dual point of the sinh-{G}ordon model",
    JOURNAL = "J. High Energy Phys.",
    YEAR = "2021",
    NUMBER = "1",
    PAGES = "Paper No. 014, 82",
    DOI = "10.1007/jhep01(2021)014",
    URL = "https://doi.org/10.1007/jhep01(2021)014"
}

@article{MR4430201,
    AUTHOR = "Bernard, D. and LeClair, A.",
    TITLE = "The sinh-{G}ordon model beyond the self dual point and the freezing transition in disordered systems",
    JOURNAL = "J. High Energy Phys.",
    YEAR = "2022",
    NUMBER = "5",
    PAGES = "Paper No. 022, 23",
    DOI = "10.1007/jhep05(2022)022",
    URL = "https://doi.org/10.1007/jhep05(2022)022"
}

@incollection{MR4680395,
    AUTHOR = "Kozlowski, K.K.",
    TITLE = "Bootstrap approach to {$1+1$}-dimensional integrable quantum field theories: the case of the sinh-{G}ordon model",
    BOOKTITLE = "I{CM}---{I}nternational {C}ongress of {M}athematicians. {V}ol. 5. {S}ections 9--11",
    PAGES = "4096--4118",
    PUBLISHER = "EMS Press, Berlin",
    YEAR = "[2023] \copyright 2023",
    ISBN = "978-3-98547-063-1; 978-3-98547-563-6; 978-3-98547-058-7",
}

@article{MR4607722,
    AUTHOR = "Kozlowski, K.K.",
    TITLE = "On convergence of form factor expansions in the infinite volume quantum {S}inh-{G}ordon model in {$1+1$} dimensions",
    JOURNAL = "Invent. Math.",
    VOLUME = "233",
    YEAR = "2023",
    NUMBER = "2",
    PAGES = "725--827",
    DOI = "10.1007/s00222-023-01192-7",
    URL = "https://doi.org/10.1007/s00222-023-01192-7"
}

@article{MR1625191,
    AUTHOR = "Fateev, V. and Lukyanov, S. and Zamolodchikov, A. and Zamolodchikov, A.",
    TITLE = "Expectation values of local fields in the {B}ullough-{D}odd model and integrable perturbed conformal field theories",
    JOURNAL = "Nuclear Phys. B",
    VOLUME = "516",
    YEAR = "1998",
    NUMBER = "3",
    PAGES = "652--674",
    DOI = "10.1016/S0550-3213(98)00002-9",
    URL = "https://doi.org/10.1016/S0550-3213(98)00002-9"
}

@Article{2408.16574,
    author = "Barashkov, N. and Oikarinen, J. and Wong, M.D.",
    title = "Small deviations of {G}aussian multiplicative chaos and the free energy of the two-dimensional massless {S}inh--{G}ordon model",
    year = "2024",
    eprint = "2408.16574",
    publisher = "arXiv",
    note = "Preprint, arXiv:2408.16574"
}

@article{MR5055747,
    AUTHOR = "Guillarmou, C. and Gunaratnam, T.S. and Vargas, V.",
    TITLE = "2d {S}inh-{G}ordon {M}odel on the {I}nfinite {C}ylinder",
    JOURNAL = "Commun. Math. Phys.",
    VOLUME = "407",
    YEAR = "2026",
    NUMBER = "5",
    PAGES = "Paper No. 97",
    DOI = "10.1007/s00220-026-05588-3",
    URL = "https://doi.org/10.1007/s00220-026-05588-3"
}

@Article{2408.16649,
    author = "Hofstetter, M. and Zeitouni, O.",
    journal = "Commun. Math. Phys.",
    title = "Decay of correlations for the massless hierarchical {L}iouville model in infinite volume",
    year = "2026+",
    note = "to appear",
    archiveprefix = "arXiv",
    eprint = "2408.16649"
}

@article{MR1144469,
    AUTHOR = "Funaki, T.",
    TITLE = "The reversible measures of multi-dimensional {G}inzburg-{L}andau type continuum model",
    JOURNAL = "Osaka J. Math.",
    VOLUME = "28",
    YEAR = "1991",
    NUMBER = "3",
    PAGES = "463--494",
    URL = "http://projecteuclid.org/euclid.ojm/1200783221"
}

@Article{2609.08980,
    author = "Abdelghani, O. and Bauerschmidt, R. and Bodineau, T. and Dagallier, B.",
    title = "The {D}ing-{S}ong-{S}un inequality for a class of even ferromagnets",
    year = "2026",
    eprint = "2609.08980",
    publisher = "arXiv",
    note = "Preprint, arXiv:2609.08980"
}

@Article{2609.17461,
    author = "Abdelghani, O. and Bauerschmidt, R. and Hofstetter, M. and Zeitouni, O.",
    title = "Mass gap for the hierarchical sinh-{G}ordon model",
    year = "2026",
    archiveprefix = "arXiv",
    eprint = "2609.17461",
    note = "Preprint, arXiv:2609.17461"
}

@Article{MR449366,
    author = "Newman, C.M.",
    journal = "J. Stat. Phys.",
    title = "Rigorous results for general {I}sing ferromagnets",
    year = "1976",
    number = "5",
    pages = "399-406",
    volume = "15",
    doi = "10.1007/BF01020342",
}

@book{MR5088010,
    AUTHOR = "Berestycki, N. and Powell, E.",
    TITLE = "Gaussian free field and {L}iouville quantum gravity",
    SERIES = "Cambridge Studies in Advanced Mathematics",
    VOLUME = "220",
    PUBLISHER = "Cambridge University Press, Cambridge",
    YEAR = "2026",
    PAGES = "xix+416",
    ISBN = "978-1-009-40550-8",
}
\bibliographystyle{plain}

\end{document}